\documentclass{bcp}
\usepackage{amsmath,amsthm}
\usepackage{amssymb}
\usepackage{latexsym}
\usepackage{enumerate}
\usepackage[colorlinks=true]{hyperref}
\hypersetup{urlcolor=blue, citecolor=red, linkcolor= black}

\theoremstyle{definition}

\newtheorem{thm}{Theorem}[section]

\newtheorem{lem}[thm]{Lemma}

\newtheorem{prop}[thm]{Proposition}

\theoremstyle{definition}
\makeatletter
\@namedef{subjclassname@2010}{%
\textup{2010} Mathematics Subject Classification}
\makeatother

\def \varpi {\bar \omega}

\def \Z{\mathbb Z}
\def \R {\mathbb R}

\def \H {\mathbb H}

\def \C {\mathbb{C}}

\def \P {\mathbb {HP}}
\def \CP {\mathbb{CP}}

\def\today{\space\ifcase\month \or January \or February \or March \or April\or May \or June \or July \or August\or September \or October\or November \or December \fi\space\number\day,\space\number\year}

\usepackage{color}

\renewcommand\theequation{\arabic{equation}}

\begin{document}

\mathclass{Primary 30G35; Secondary 30D30.}

% 30ÐXX FUNCTIONS OF A COMPLEX VARIABLE {For analysis on manifolds, see 58ÐXX}
%
% 30Dxx Entire and meromorphic functions, and related topics
% 30D30 Meromorphic functions, general theory
%
% 30Gxx Generalized function theory
% 30G35 Functions of hypercomplex variables and generalized variables

\abbrevauthors{P. Dolbeault}
\abbrevtitle{On a noncommutative algebraic geometry}

\title{On a noncommutative algebraic geometry}

\author{Pierre Dolbeault}
\address{Sorbonne Universit\'es,\\ UPMC Univ.~Paris 06,\\ Institut de Math\'ematiques de Jussieu - Paris Rive Gauche UMR7586,\\ 4, place Jussieu 75005 Paris, France\\ E-mail: pierre.dolbeault@upmc.fr}

\maketitlebcp

%%%%%%%%%%%%%%%%%%%%%%%%%%%%%%%%%%%%%%%%%%%%%%%%%%%%%%%%%
%%%%%%%%%%%%%%%%%%%%%%%%%%%%%%%%%%%%%%%%%%%%%%%%%%%%%%%%%
\begin{abstract}
Several sets of quaternionic functions are described and studied with respect to hyperholomorphy, addition and (non commutative) multiplication, on open sets of $\H$, then Hamilton 4-manifolds analogous to Riemann surfaces, for $\H$ instead of $\C$, are defined, and so begin to describe a class of four dimensional manifolds.
\end{abstract}

\thispagestyle{empty}
{\small\tableofcontents}

%%%%%%%%%%%%%%%%%%%%%%%%%%%%%%%%%%%%%%%%%%%%%%%%%%%%%%%%%%%%%%%%%%%%%%
%%%%%%%%%%%%%%%%%%%%%%%%%%%%%%%%%%%%%%%%%%%%%%%%%%%%%%%%%%%%%%%%%%%%%%
\section{Introduction}\label{Sec:Intro}

We first recall the definition of the field $\H$ of quaternions using pairs of complex numbers and a modified Cauchy-Fueter operator (section~\ref{Sec2})
that have been introduced by C. Colombo and al., \cite{CLSSS 07}. We will only use right multiplication. We will consider $C^\infty$ $\H$-valued quaternionic functions defined on an open set $U$ of $\H$ whose behavior mimics the behavior of holomorphic functions on an open set of $\C$. If such a function does not vanish identically, it has an (algebraic) inverse. Finally we describe properties of Hyperholomorphic functions with respect to addition and multiplication.

In section~\ref{Sec3}, we characterize the quaternionic functions which are, almost everywhere, hyperholomorphic and whose inverses are hyperlomorphic almost everywhere, on $U$, as the solutions of a system of two non linear PDE. We find non trivial examples of a solution showing that the considered space of functions is significant; we will call these functions Hypermeromorphic.

At the moment, I am unable to get the general solution of the the system of PDE. Same difficulty for subsequent occurring systems of PDE. 

%---------

In section~\ref{Sec4}, we describe a subspace of hyperholomorphic and hypermeromorphic functions defined almost everywhere on $U$, having ``good properties for addition and multiplication"; we again obtain systems of non linear PDE.

In section~\ref{Sec5} and the following, we consider globalization of the above notions, define Hamilton 4-manifolds analogous to Riemann surfaces, for $\H$ instead of $\C$, and give examples of such manifolds; our ultimate aim is to describe a class of 4-dimensional manifolds.

%%%%%%%%%%%%%%%%%%%%%%%%%%%%%%%%%%%%%%%%%%%%%%%%%%%%%%%%%%%%%%%%%%%%%%
%%%%%%%%%%%%%%%%%%%%%%%%%%%%%%%%%%%%%%%%%%%%%%%%%%%%%%%%%%%%%%%%%%%%%%
\section{Quaternions. \texorpdfstring{$\H$-}{H-}valued functions. Hyperholomorphic functions}\label{Sec2}
\hspace*{1pt}\\
See~\cite{CSSS 04,CLSSS 07,D13}.

%%%%%%%%%%%%%%%%%%%%%%%%%%%%%%%%%%%%%%%%%%%%%%%%%%%%%%%%%%%%%%%%%%%%%%
\subsection{Quaternions}

If $q\in\H$, then $q=z_1+z_2{\bf j}$ where $z_1, z_2\in\C$. We have $z_1{\bf j}={\bf j}\overline z_1$, and note $\vert q\vert= \vert z_1\vert^2+\vert z_2\vert^2$.

The conjugate of $q$ is $\overline q=\overline z_1-z_2{\bf j}.$
Let us denote * the (right) multiplication in $\H$, then the right inverse of $q$ is: $q^{-1}=\vert q\vert^{-1}\overline q$ 
%%%%%%%%%%%%%%%%%%%%%%%%%%%%%%%%%%%%%%%%%%%%%%%%%%%%%%%%%%%%%%%%%%%%%%
\subsection{Quaternionic functions}

Let $U$ be an open set of $\H\cong \C^2$ and $f\in C^\infty(U,\H)$, then $f=f_1+f_2{\bf j}$, where $f_1, f_2\in C^\infty(U,\C)$. The complex valued functions $f_1, f_2$ will be called the {\it components} of $f$.

%%%%%%%%%%%%%%%%%%%%%%%%%%%%%%%%%%%%%%%%%%%%%%%%%%%%%%%%%%%%%%%%%%%%%%
\subsection{Definitions}

Let $U$ be an open neighborhood of $0$ in $\H\cong\C^2.$

{\it\noindent (a) From now on, we will consider the quaternionic functions $f=f_1+f_2{\bf j}$ having the following properties}:

(i) When $f_1$ and $f_2$ are not holomorphic, the set $Z(f_1)\cap
Z(f_2)$ is discrete on $U$;

(ii) for every $q\in Z(f_1)\cap Z(f_2)$, $J_q^\alpha$ (.) denoting the {\it jet of order $\alpha$ at $q$} (see \cite{M 66}), let $m_i=\sup_{\alpha_i} J_q^{\alpha_i}(f_i)=0$; $m_i, i=1,2,$ is finite.

$m_q=inf m_i$ is the {\it order of the zeroe $q$ of $f$}.

{\it\noindent (b) We will also consider the quaternionic functions defined almost everywhere on} $U$ (i.e. outside a locally finite set of $C^\infty$ hypersurfaces, namely $Z(f_1),
Z(f_2)$).

%%%%%%%%%%%%%%%%%%%%%%%%%%%%%%%%%%%%%%%%%%%%%%%%%%%%%%%%%%%%%%%%%%%%%%
\subsection{Modified Cauchy-Fueter operator ${\mathcal D}$. Hyperholomorphic functions}
\hspace*{1pt}\\
See \cite{CLSSS 07,F 39}.

For $f\in C^\infty (U,\H)$, with $f=f_1+f_2{\bf j}$, 

${\mathcal D}f(q) =\displaystyle\frac{1}{2}\big(\frac{\partial}{\partial\overline z_1}+{\bf j}\frac{\partial}{\partial\overline z_2}\big)f(q)$.
\\
A function $f\in C^\infty (U,\H)$ is said {\it hyperholomorphic} if ${\mathcal D} f=0$.
\\
Characterization of the hyperholomorphic function $f$ on $U$:
\begin{equation}\label{Eq1}
{\frac{\partial
f_1}{\partial\overline z_1}-\frac{\partial\overline f_2}{\partial z_2}}=0; \ {\displaystyle\frac{\partial f_1}{\partial\overline
z_2}+\frac{\partial\overline f_2}{\partial z_1}}=0,\ {\rm on} \ U.
\end{equation}

%%%%%%%%%%%%%%%%%%%%%%%%%%%%%%%%%%%%%%%%%%%%%%%%%%%%%%%%%%%%%%%%%%%%%%
\subsection{Several families of meromorphic functions}

The conditions $f_1$ is holomorphic and $f_2$ is holomorphic are equivalent on $U$; the same is true for almost everywhere defined holomorphic functions on $U$. 

By definition, {\it holomorphic (almost everywhere defined functions of two complex variables} on $U$ are such that $f_2=0$, and $f_1$ is (almost everywhere) holomorphic.

%%%%%%%%%%%%%%%%%%%%%%%%%%%%%%%%%%%%%%%%%%%%%%%%%%%%%%%%%%%%%%%%%%%%%%
\subsubsection{Consider the almost everywhere defined hyperholomorphic functions on $U$ whose components are real}

$$f=f_1+f_2{\bf j}$$ 

According to a remark of Guy Roos in March 2013, they are almost everywhere holomorphic \cite{R 13}. 

%%%%%%%%%%%%%%%%%%%%%%%%%%%%%%%%%%%%%%%%%%%%%%%%%%%%%%%%%%%%%%%%%%%%%%
\subsubsection{The above considered almost everywhere holomorphic functions are meromorphic and constitute two $\H$-commutative algebras $A_1, A_2$, with common origin $0$}\label{Subsec:2.5.2}

Let $f = a+b{\bf i}$, and $g = c+d{\bf j}$, with $a,b,c,d \in \R$ be two almost everywhere defined holomorphic functions i.e.meromorphic functions on $U$. 

$A_1$ is the set of the meromorphic functions $f= a+b{\bf i}$, and $A_2$ is the set of meromorphic functions $g=c+d{\bf j}$, with $a,b,c,d \in \R$

The sums $f+g=a+c+d{\bf j}+b{\bf i}$ constitute the algebra $A_1+A_2$ of meromorphic functions.

More generally, $A_{\alpha, \beta} = \alpha A_1+\beta A_2$, with $\alpha, \beta \in \R$ is an algebra of meromorphic functions on $U$.

$A_{\alpha\beta}=\displaystyle\sum_{a,b,c,d,\alpha,\beta \in\R}\alpha(a+b{\bf i})+\beta(c+d{\bf j})$

%%%%%%%%%%%%%%%%%%%%%%%%%%%%%%%%%%%%%%%%%%%%%%%%%%%%%%%%%%%%%%%%%%%%%%
\subsubsection{We now begin to introduce multiplication for hyperholomorphic functions, addition and scalar muliplication being obvious}

%%%%%%%%%%%%%%%%%%%%%%%%%%%%%%%%%%%%%%%%%%%%%%%%%%%%%%%%%%%%%%%%%%%%%%
\subsection{Multiplication of almost everywhere defined hyperholomorphic functions}

%---------------------------------------------------------------------
\begin{prop}{\it \label{Proposition 2.3.2} Let $f'$, $f''$ be two almost everywhere defined hyperholomorphic functions. Then, their product $f'*f''$ satisfies:
$$
{\mathcal D} (f'*f'')={\mathcal D}f'*{\bf j}f''+\big(f'(\frac{\partial }{\partial\overline z_1})+\overline f'{\bf j}\frac{\partial }{\partial\overline z_2}\big)f''
$$
}\end{prop}
%---------------------------------------------------------------------
\begin{proof}
$f'=f'_1+f'_2{\bf j}$, $f''=f''_1+f''_2{\bf j}$ be two hyperholomorphic functions. 
\medskip

We have: $f'*f''=(f'_1+f'_2{\bf j})(f''_1+f''_2{\bf j})
= f'_1f''_1- f'_2\overline f''_2+(f'_1f''_2+f'_2\overline f''_1){\bf j}$
\smallskip

\noindent Compute
$$
\frac{1}{2}\big(\frac{\partial}{\partial\overline z_1}
+{\bf j}\frac{\partial}{\partial\overline z_2}\big) \big(f'_1f''_1- f'_2\overline f''_2+(f'_1f''_2+f'_2\overline f''_1){\bf j}\big)
$$
By derivation of the first factors of the sum $f'*f''$, we get the first term: 
$$
\frac{1}{2}\big(\frac{\partial f'_1}{\partial\overline z_1}
+{\bf j}\frac{\partial f'_1}{\partial\overline z_2}\big)(f''_1+f''_2{\bf j}) +\frac{1}{2}\big(\frac{\partial f'_2}{\partial\overline z_1}
+{\bf j}\frac{\partial f'_2}{\partial\overline z_2}\big) {\bf j}{\bf j}(\overline f''_2-\overline f''_1{\bf j}) 
$$
$$
=\frac{1}{2}\big(\frac{\partial f'_1}{\partial\overline z_1}
+{\bf j}\frac{\partial f'_1}{\partial\overline z_2}\big)(f''_1+f''_2{\bf j}) +\frac{1}{2}\big(\frac{\partial f'_2{\bf j}}{\partial\overline z_1}
+{\bf j}\frac{\partial f'_2{\bf j}}{\partial\overline z_2}\big) {\bf j}(f''_2{\bf j}+f''_1)= {\mathcal D}f'*{\bf j}f''
$$
By derivation in 
$$
\frac{1}{2}\big(\frac{\partial}{\partial\overline z_1}
+{\bf j}\frac{\partial}{\partial\overline z_2}\big) \big(f'_1f''_1+f'_2{\bf j} f''_2{\bf j}+(f'_1f''_2{\bf j}+f'_2{\bf j}f''_1)\big)
$$
of the second factors of the sum $f'*f''$, we get the second term (up to factor~$\frac{1}{2}$):

\begin{multline*}
f'_1\frac{\partial f''_1}{\partial\overline z_1}+\overline f'_1{\bf j}\frac{\partial f''_1}{\partial\overline z_2}+f'_1\frac{\partial f''_2}{\partial\overline z_1}{\bf j}+\overline f'_1{\bf j}\frac{\partial f''_2}{\partial\overline z_2}{\bf j}
\\
+f'_2{\bf j}\frac{\partial f''_2}{\partial\overline z_1}{\bf j}+\overline f'_2{\bf j}\frac{\partial f''_2}{\partial\overline z_2}+f'_2{\bf j}\frac{\partial f''_1}{\partial\overline z_1}+\overline f'_2{\bf j}{\bf j}\frac{\partial f''_1}{\partial\overline z_2}
\\
=(f'_1+f'_2{\bf j})(\frac{\partial }{\partial\overline z_1})(f''_1+ f''_2{\bf j}) +(\overline f'_1+\overline f'_2{\bf j}){\bf j}\frac{\partial }{\partial\overline z_2}(f''_1+ f''_2{\bf j})
\\
= \big((f'_1+f'_2{\bf j})(\frac{\partial }{\partial\overline z_1})+(\overline f'_1+\overline f'_2{\bf j}){\bf j}\frac{\partial }{\partial\overline z_2}\big)(f''_1+ f''_2{\bf j})
\\
= \big(f'(\frac{\partial }{\partial\overline z_1})+\overline f'{\bf j}\frac{\partial }{\partial\overline z_2}\big)f''.
\end{multline*}
\end{proof} 

%%%%%%%%%%%%%%%%%%%%%%%%%%%%%%%%%%%%%%%%%%%%%%%%%%%%%%%%%%%%%%%%%%%%%%
%%%%%%%%%%%%%%%%%%%%%%%%%%%%%%%%%%%%%%%%%%%%%%%%%%%%%%%%%%%%%%%%%%%%%%
\section{Almost everywhere hyperholomorphic functions whose inverses are almost everywhere hyperholomorphic}\label{Sec3}
%%%%%%%%%%%%%%%%%%%%%%%%%%%%%%%%%%%%%%%%%%%%%%%%%%%%%%%%%%%%%%%%%%%%%%
\subsection{Definitions}

We call {\it inverse} of a quaternionic function $f: q\mapsto f(q)$, the function defined almost everywhere on $U$: $q\mapsto f(q)^{-1}$; then: $f^{-1}=\vert f\vert^{-1}\overline f$, where $\overline f$ is the (quaternionic) conjugate of $f$, then: $f^{-1}=\vert f\vert^{-1}(\overline f_1-f_2{\bf j}).$ 

Behavior of $f^{-1}$ at $q\in Z(f)$. Let $n_1= sup J_q^\alpha (\vert f\vert\overline f_1^{-1})$; $n_2= sup J_q^\alpha (\vert f\vert\overline f_2^{-1})$.

Define : $n_q= sup \ n_i$, $i=1,2$ as the {\it order of the pole $q$ of $f^{-1}$.}

%%%%%%%%%%%%%%%%%%%%%%%%%%%%%%%%%%%%%%%%%%%%%%%%%%%%%%%%%%%%%%%%%%%%%%
\subsection{Characterisation}

%---------------------------------------------------------------------
\begin{prop}{\label{Proposition 3.1} \it The following conditions are equivalent:

(i) the function $f$ and its right inverse are hyperholomorphic, when they are defined;

(ii) we have the equations: 

$$(\overline f_1-f_1)\frac{\partial\overline f_1}{\partial z_1
}-\overline f_2\frac{\partial f_2}{\partial z_1} -f_2\frac{\partial\overline f_1}{\partial\overline z_2}=0$$

$$\overline f_2\frac{\partial f_1}{\partial z_1}+\frac{\partial \overline f_2}{\partial z_1}(\overline f_1-f_1)-f_2\frac{\partial\overline f_2}{\partial \overline z_2}=0$$
}\end{prop}
%---------------------------------------------------------------------
\begin{proof}
Let $f=f_1+f_2{\bf j}$ be a hyperholomorphic function and $g=g_1+g_2{\bf j}=\displaystyle \vert f\vert^{-1}(\overline f_1-f_2{\bf j})$ its inverse; so $g_1=\vert f\vert^{-1}\overline f_1$; $g_2=-\vert f\vert^{-1}f_2$, 
where $\vert f\vert=(f_1\overline f_1+f_2\overline f_2)$.
\begin{eqnarray*}
{\mathcal D}g(q)&=&\frac{1}{2}\big(\frac{\partial}{\partial\overline z_1}
+{\bf j}\frac{\partial}{\partial\overline z_2}\big) g(q)
=\frac{1}{2}\big({\frac{\partial
g_1}{\partial\overline z_1}-\frac{\partial\overline g_2}{\partial z_2}}\big)
(q)+{\bf j}\frac{1}{2}\big({\frac{\partial g_1}{\partial\overline
z_2}+\frac{\partial\overline g_2}{\partial z_1}}\big) (q)
\\
\frac{\partial g_1}{\partial\overline z_1}&=& \vert f\vert^{-1}\frac{\partial\overline f_1}{\partial \overline z_1}-\vert f\vert^{-2} \overline f_1\big(\frac{\partial f_1}{\partial \overline z_1}\overline f_1+f_1\frac{\partial\overline f_1}{\partial\overline z_1}+\frac{\partial f_2}{\partial \overline z_1}\overline f_2+f_2\frac{\partial\overline f_2}{\partial \overline z_1}\big)
\\
-\frac{\partial\overline g_2}{\partial z_2}&=& \vert f\vert^{-1}\frac{\partial\overline f_2}{\partial z_2}-\vert f\vert^{-2} \overline f_2\big(\frac{\partial f_1}{\partial z_2}\overline f_1+f_1\frac{\partial\overline f_1}{\partial z_2}+\frac{\partial f_2}{\partial z_2}\overline f_2+f_2\frac{\partial\overline f_2}{\partial z_2}\big)
\\
\frac{\partial g_1}{\partial\overline z_2}&=& \vert f\vert^{-1}\frac{\partial\overline f_1}{\partial \overline z_2}-\vert f\vert^{-2} \overline f_1\big(\frac{\partial f_1}{\partial \overline z_2}\overline f_1+f_1\frac{\partial\overline f_1}{\partial \overline z_2}+\frac{\partial f_2}{\partial \overline z_2}\overline f_2+f_2\frac{\partial\overline f_2}{\partial \overline z_2}\big)
\\
\frac{\partial\overline g_2}{\partial z_1}&=& - \vert f\vert^{-1}\frac{\partial\overline f_2}{\partial z_1}+\vert f\vert^{-2} \overline f_2\big(\frac{\partial f_1}{\partial z_1}\overline f_1+f_1\frac{\partial\overline f_1}{\partial z_1}+\frac{\partial f_2}{\partial z_1}\overline f_2+f_2\frac{\partial\overline f_2}{\partial z_1}\big)
\end{eqnarray*}
\begin{eqnarray*}
&&\hspace*{-36pt}2\vert f\vert^2{\mathcal D}g
\\
&=&(f_1\overline f_1+f_2\overline f_2)(\frac{\partial\overline f_1}{\partial \overline z_1}+\frac{\partial\overline f_2}{\partial z_2})-\overline f_1 f_1\frac{\partial \overline f_1}{\partial \overline z_1}-\overline f_1\overline f_1\frac{\partial f_1}{\partial \overline z_1}-\overline f_1f_2\frac{\partial\overline f_2}{\partial \overline z_1}-\overline f_1\overline f_2\frac{\partial f_2}{\partial \overline z_1}%\leqno (2) 
\\
&&-\overline f_1\overline f_2\frac{\partial f_1}{\partial z_2}- f_1\overline f_2\frac{\partial\overline f_1}{\partial z_2}-
\overline f_2\overline f_2\frac{\partial f_2}{\partial z_2}- f_2\overline f_2\frac{\partial\overline f_2}{\partial z_2}
\\
&&+{\bf j}\Big((f_1\overline f_1+f_2\overline f_2)(\frac{\partial\overline f_1}{\partial \overline z_2}-\frac{\partial\overline f_2}{\partial z_1})-\overline f_1\overline f_1\frac{\partial f_1}{\partial \overline z_2}-\overline f_1 f_1\frac{\partial \overline f_1}{\partial \overline z_2}-\overline f_1\overline f_2\frac{\partial f_2}{\partial \overline z_2}-\overline f_1 f_2\frac{\partial\overline f_2}{\partial \overline z_2}
\\
&&+\overline f_1\overline f_2\frac{\partial f_1}{\partial z_1}+ f_1\overline f_2\frac{\partial\overline f_1}{\partial z_1}+
\overline f_2\overline f_2\frac{\partial f_2}{\partial z_1}+ f_2\overline f_2\frac{\partial\overline f_2}{\partial z_1}\Big)\end{eqnarray*}
Use the fact: $f$ is hyperholomorphic: 
$$
{\displaystyle\frac{\partial
f_1}{\partial\overline z_1}-\frac{\partial\overline f_2}{\partial z_2}}=0; \ {\displaystyle\frac{\partial f_1}{\partial\overline
z_2}+\frac{\partial\overline f_2}{\partial z_1}}=0 \leqno{\eqref{Eq1}}
$$ 
\begin{multline*}
2\vert f\vert^2{\mathcal D}g=\\ f_1\overline f_1\frac{\partial\overline f_2}{\partial z_2} +f_2\overline f_2\frac{\partial\overline f_1}{\partial \overline z_1} -\overline f_1\overline f_1\frac{\partial f_1}{\partial \overline z_1}-\overline f_1f_2\frac{\partial\overline f_2}{\partial \overline z_1}-\overline f_1\overline f_2\frac{\partial f_1}{\partial z_2}-
\overline f_2\overline f_2\frac{\partial f_2}{\partial z_2}+\overline f_2\frac{\partial f_2}{\partial \overline z_1}(f_1-\overline f_1)+%\leqno (3)
\\+{\bf j}\Big(+f_2\overline f_2\frac{\partial\overline f_1}{\partial \overline z_2}-\overline f_1\overline f_2\frac{\partial f_2}{\partial \overline z_2}-\overline f_1 f_2\frac{\partial\overline f_2}{\partial \overline z_2}+\overline f_1\frac{\partial f_1}{\partial \overline z_2}(f_1-\overline f_1)
+\overline f_1\overline f_2\frac{\partial f_1}{\partial z_1}+ f_1\overline f_2\frac{\partial\overline f_1}{\partial z_1}+
\overline f_2\overline f_2\frac{\partial f_2}{\partial z_1}\Big)
\end{multline*}

$f$ being hyperholomorphic, $g$ hyperholomorphic is equivalent to the system of two equations:

$$
+f_1\overline f_1\frac{\partial\overline f_2}{\partial z_2} +f_2\overline f_2\frac{\partial\overline f_1}{\partial \overline z_1} -\overline f_1\overline f_1\frac{\partial f_1}{\partial \overline z_1}-\overline f_1f_2\frac{\partial\overline f_2}{\partial \overline z_1}-\overline f_1\overline f_2\frac{\partial f_1}{\partial z_2}-
\overline f_2\overline f_2\frac{\partial f_2}{\partial z_2}+\overline f_2\frac{\partial f_2}{\partial \overline z_1}(f_1-\overline f_1)=0%\leqno (4)
$$
$$
+f_2\overline f_2\frac{\partial\overline f_1}{\partial \overline z_2}-\overline f_1\overline f_2\frac{\partial f_2}{\partial \overline z_2}-\overline f_1 f_2\frac{\partial\overline f_2}{\partial \overline z_2}+\overline f_1\frac{\partial f_1}{\partial \overline z_2}(f_1-\overline f_1)
+\overline f_1\overline f_2\frac{\partial f_1}{\partial z_1}+ f_1\overline f_2\frac{\partial\overline f_1}{\partial z_1}+
\overline f_2\overline f_2\frac{\partial f_2}{\partial z_1}=0
$$

$f_1$ and $f_2$ satisfy, by conjugation of the second equation:
$$ +f_2\overline f_2\frac{\partial f_1}{\partial z_1} - f_1 f_1\frac{\partial\overline f_1}{\partial z_1}-f_1\overline f_2\frac{\partial f_2}{\partial z_1} + f_2\frac{\partial\overline f_2}{\partial z_1}(\overline f_1-f_1) +f_1\overline f_1\frac{\partial f_2}{\partial\overline z_2}-f_1f_2\frac{\partial \overline f_1}{\partial\overline z_2}-f_2f_2\frac{\partial\overline f_2}{\partial\overline z_2}=0
%\leqno (5)
$$
$$
+\overline f_1\overline f_2\frac{\partial f_1}{\partial z_1}+ f_1\overline f_2\frac{\partial\overline f_1}{\partial z_1}.+
\overline f_2\overline f_2\frac{\partial f_2}{\partial z_1}+f_2\overline f_2\frac{\partial\overline f_1}{\partial \overline z_2}-\overline f_1\overline f_2\frac{\partial f_2}{\partial \overline z_2}-\overline f_1 f_2\frac{\partial\overline f_2}{\partial \overline z_2}+\overline f_1\frac{\partial f_1}{\partial \overline z_2}(f_1-\overline f_1)
=0
$$ 
Using $\eqref{Eq1}$, we get: 
$$
+f_2\overline f_2\frac{\partial f_1}{\partial z_1} + f_1 (\overline f_1- f_1)\frac{\partial\overline f_1}{\partial z_1}-f_1\overline f_2\frac{\partial f_2}{\partial z_1} + f_2\frac{\partial\overline f_2}{\partial z_1}(\overline f_1-f_1) -f_1f_2\frac{\partial \overline f_1}{\partial\overline z_2}-f_2f_2\frac{\partial\overline f_2}{\partial\overline z_2}=0
%\leqno (6) 
$$
$$
+\overline f_1\overline f_2\frac{\partial f_1}{\partial z_1}+ (f_1-\overline f_1)\overline f_2\frac{\partial\overline f_1}{\partial z_1}+
\overline f_2\overline f_2\frac{\partial f_2}{\partial z_1}+\overline f_1\frac{\partial f_1}{\partial \overline z_2}(f_1-\overline f_1)+f_2\overline f_2\frac{\partial\overline f_1}{\partial \overline z_2}.
-\overline f_1 f_2\frac{\partial\overline f_2}{\partial \overline z_2}=0
$$

Assume $f_1\not =0$, $f_2\not =0$

$$
\overline f_1 \big(f_2\overline f_2\frac{\partial f_1}{\partial z_1} + f_1 (\overline f_1- f_1)\frac{\partial\overline f_1}{\partial z_1}-f_1\overline f_2\frac{\partial f_2}{\partial z_1} + f_2\frac{\partial\overline f_2}{\partial z_1}(\overline f_1-f_1) -f_1f_2\frac{\partial \overline f_1}{\partial\overline z_2}-f_2f_2\frac{\partial\overline f_2}{\partial\overline z_2}\big)=0
%\leqno (7) 
$$
$$
-f_2\big(+\overline f_1\overline f_2\frac{\partial f_1}{\partial z_1}+ (f_1-\overline f_1)\overline f_2\frac{\partial\overline f_1}{\partial z_1}+
\overline f_2\overline f_2\frac{\partial f_2}{\partial z_1}
-\overline f_1\frac{\partial\overline f_2}{\partial z_1}(f_1-\overline f_1)+f_2\overline f_2\frac{\partial\overline f_1}{\partial \overline z_2}-\overline f_1 f_2\frac{\partial\overline f_2}{\partial\overline z_2}\big)=0
$$ 

By sum:
$$
\overline f_1 \big( f_1 (\overline f_1- f_1)\frac{\partial\overline f_1}{\partial z_1}-f_1\overline f_2\frac{\partial f_2}{\partial z_1} + f_2\frac{\partial\overline f_2}{\partial z_1}(\overline f_1-f_1) -f_1f_2\frac{\partial \overline f_1}{\partial\overline z_2}\big)
$$
$$
-f_2\big((f_1-\overline f_1)\overline f_2\frac{\partial\overline f_1}{\partial z_1}+
\overline f_2\overline f_2\frac{\partial f_2}{\partial z_1}
-\overline f_1\frac{\partial\overline f_2}{\partial z_1}(f_1-\overline f_1)+f_2\overline f_2\frac{\partial\overline f_1}{\partial \overline z_2}\big)=0
$$
i.e. 
$$
(\overline f_1f_1+f_2\overline f_2)\big ( (\overline f_1- f_1)\frac{\partial\overline f_1}{\partial z_1}-\overline f_2\frac{\partial f_2}{\partial z_1}-f_2\frac{\partial\overline f_1}{\partial \overline z_2}\big )=0
$$

$$
\overline f_2 \big(f_2\overline f_2\frac{\partial f_1}{\partial z_1} + f_1 (\overline f_1- f_1)\frac{\partial\overline f_1}{\partial z_1}-f_1\overline f_2\frac{\partial f_2}{\partial z_1} + f_2\frac{\partial\overline f_2}{\partial z_1}(\overline f_1-f_1) -f_1f_2\frac{\partial \overline f_1}{\partial\overline z_2}-f_2f_2\frac{\partial\overline f_2}{\partial\overline z_2}\big)=0
%\leqno (8) 
$$
$$
f_1\big(+\overline f_1\overline f_2\frac{\partial f_1}{\partial z_1}+ (f_1-\overline f_1)\overline f_2\frac{\partial\overline f_1}{\partial z_1}+
\overline f_2\overline f_2\frac{\partial f_2}{\partial z_1}
-\overline f_1\frac{\partial\overline f_2}{\partial z_1}(f_1-\overline f_1)+f_2\overline f_2\frac{\partial\overline f_1}{\partial \overline z_2}-\overline f_1 f_2\frac{\partial\overline f_2}{\partial\overline z_2}\big)=0
$$ 

By sum
$$
\overline f_2 \big(f_2\overline f_2\frac{\partial f_1}{\partial z_1} + f_2\frac{\partial\overline f_2}{\partial z_1}(\overline f_1-f_1) -f_2f_2\frac{\partial\overline f_2}{\partial\overline z_2}\big) 
$$
$$
+f_1\big(\overline f_1\overline f_2\frac{\partial f_1}{\partial z_1}
-\overline f_1\frac{\partial\overline f_2}{\partial z_1}(f_1-\overline f_1)-\overline f_1 f_2\frac{\partial\overline f_2}{\partial\overline z_2}\big)=0
$$ 
i.e.
$$
\overline f_2\frac{\partial f_1}{\partial z_1}+ \frac{\partial\overline f_2}{\partial z_1}(\overline f_1-f_1)-f_2\frac{\partial\overline f_2}{\partial\overline z_2}=0 
$$
\end{proof}
%%%%%%%%%%%%%%%%%%%%%%%%%%%%%%%%%%%%%%%%%%%%%%%%%%%%%%%%%%%%%%%%%%%%%%
\subsection{Definition} We will call {\it w-hypermeromorphic function} (w- for {\it weak}) any almost everywhere defined hyperholomorphic function whose right inverse is hyperholomorphic almost everywhere.

%%%%%%%%%%%%%%%%%%%%%%%%%%%%%%%%%%%%%%%%%%%%%%%%%%%%%%%%%%%%%%%%%%%%%%
%%%%%%%%%%%%%%%%%%%%%%%%%%%%%%%%%%%%%%%%%%%%%%%%%%%%%%%%%%%%%%%%%%%%%%
\newpage\section{On the spaces of hypermeromorphic functions}\label{Sec4}

%%%%%%%%%%%%%%%%%%%%%%%%%%%%%%%%%%%%%%%%%%%%%%%%%%%%%%%%%%%%%%%%%%%%%%
%%%%%%%%%%%%%%%%%%%%%%%%%%%%%%%%%%%%%%%%%%%%%%%%%%%%%%%%%%%%%%%%%%%%%%
\subsection{Sum of two w-hypermeromorphic functions}

%---------------------------------------------------------------------
\begin{prop}{\it If $f$ and $g$ are two w-hypermeromorphic functions, then the following conditions are equivalent:
\begin{enumerate}
\item[$(i)$] the sum $h=f+g$ is w-hypermeromorphic; 
\item[$(ii)$] $h$ satisfies the following PDE:
$$
-\big( \frac{\partial \vert h\vert}{\partial\overline z_1}+{\bf j}\frac{\partial\vert h\vert}{\partial\overline z_2}\big) (\overline h_1-h_2{\bf j})+\vert h\vert {\big(\frac{\partial}{\partial\overline z_1}
+{\bf j}\frac{\partial}{\partial\overline z_2}\big)}(\overline h_1-h_2{\bf j})=0
$$
\end{enumerate}
}\end{prop}
%---------------------------------------------------------------------
\begin{proof}Explicit the condition: $${\vert h\vert}^2{\mathcal D}(h^{-1})=-{\mathcal D}(\vert h)\vert)(\overline {h})+\vert h\vert {\mathcal D}(\overline {h})=0;$$with $\overline h=\overline h_1-h_2{\bf j}$

$$
2{\mathcal D}\overline h=\big(\frac{\partial}{\partial\overline z_1}
+{\bf j}\frac{\partial}{\partial\overline z_2}\big)(\overline h_1-h_2{\bf j})= \frac{\partial\overline h_1}{\partial\overline z_1}+\frac{\partial \overline h_2}{\partial z_2}-\big(\frac{\partial h_2}{\partial\overline z_1}-\frac{\partial h_1}{\partial z_2}\big){\bf j}
$$

\begin{multline*}
{\mathcal D}(\vert h\vert)={\mathcal D}(h_1\overline h_1+h_2\overline h_2)=\frac{1}{2}\big(\frac{\partial}{\partial\overline z_1}
+{\bf j}\frac{\partial}{\partial\overline z_2}\big)(h_1\overline h_1+h_2\overline h_2)
\\
=
\frac{1}{2}\big(\overline h_1\frac{\partial h_1}{\partial\overline z_1} +\overline h_2\frac{\partial h_2}{\partial\overline z_1} +h_1\frac{\partial\overline h_1}{\partial\overline z_1} +h_2\frac{\partial\overline h_2}{\partial\overline z_1}\big)
\\
+\frac{1}{2}\big(\overline h_1\frac{\partial h_1}{\partial z_2} +\overline h_2\frac{\partial h_2}{\partial z_2}+h_1\frac{\partial\overline h_1}{\partial z_2} +h_2\frac{\partial\overline h_2}{\partial z_2}\big){\bf j}=0.
\end{multline*}\end{proof}

%%%%%%%%%%%%%%%%%%%%%%%%%%%%%%%%%%%%%%%%%%%%%%%%%%%%%%%%%%%%%%%%%%%%%%
\subsection{Product of two w-hypermeromorphic functions}

%---------------------------------------------------------------------
\begin{prop}{\it Let $f$, $g$ be two w-hypermeromorphic functions on $U$, then the following conditions are equivalent:
\begin{enumerate}
\item[$(i)$] the product $f*g$ is w-hypermeremorphic;
\item[$(ii)$] $f$ and $g$ satisfy the system of PDE:

$$
g_1(\frac{\partial
f_1}{\partial\overline z_1}+\frac{\partial\overline f_2}{\partial z_2})+(f_1-\overline f_1)\frac{\partial g_1}{\partial\overline z_1}+\overline f_2\frac{\partial g_1}{\partial z_2}-f_2\frac{\partial\overline g_2}{\partial\overline z_1}=0
$$

$$
g_1(\frac{\partial
f_1}{\partial\overline z_2}-\frac{\partial\overline f_2}{\partial z_1})+(f_1-\overline f_1)\frac{\partial g_1}{\partial\overline z_2}-\overline f_2\frac{\partial g_1}{\partial z_1}-f_2\frac{\partial\overline g_2}{\partial\overline z_2}=0
$$
\end{enumerate}
}\end{prop}
%---------------------------------------------------------------------

\begin{proof}
Let $f=f_1+f_2{\bf j}$ and $g=g_1+g_2{\bf j}$ two hypermeromorphic functions and $f*g=f_1g_1-f_2\overline g_2+(f_1g_2-f_2\overline g_1){\bf j}$ their product, then
\begin{eqnarray*}
&&\frac{\partial f_1}{\partial\overline z_1}-\frac{\partial\overline f_2}{\partial z_2}=0;
\\
&&\frac{\partial
(f_1g_1-f_2\overline g_2)}{\partial\overline z_1}-\frac{\partial(\overline f_1\overline g_2-\overline f_2g_1)}{\partial z_2}\\
&&\qquad=g_1(\frac{\partial
f_1}{\partial\overline z_1}+\frac{\partial\overline f_2}{\partial z_2})-\overline g_2(\frac{\partial\overline f_1}{\partial z_2}+\frac{\partial f_2}{\partial\overline z_1})+f_1\frac{\partial g_1}{\partial\overline z_1}-\overline f_1\frac{\partial\overline g_2}{\partial z_2}+\overline f_2\frac{\partial g_1}{\partial z_2}-f_2\frac{\partial\overline g_2}{\partial\overline z_1}=0
\\
&&g_1(\frac{\partial
f_1}{\partial\overline z_1}+\frac{\partial\overline f_2}{\partial z_2})+f_1\frac{\partial g_1}{\partial\overline z_1}-\overline f_1\frac{\partial\overline g_2}{\partial z_2}+\overline f_2\frac{\partial g_1}{\partial z_2}-f_2\frac{\partial\overline g_2}{\partial\overline z_1}=0.
\\
&&\frac{\partial
(f_1g_1-f_2\overline g_2)}{\partial\overline z_2}+\frac{\partial(\overline f_1\overline g_2-\overline f_2g_1)}{\partial z_1}\\&&\qquad =g_1(\frac{\partial
f_1}{\partial\overline z_2}-\frac{\partial\overline f_2}{\partial z_1})+\overline g_2(\frac{\partial\overline f_1}{\partial z_1}-\frac{\partial f_2}{\partial\overline z_2})+f_1\frac{\partial g_1}{\partial\overline z_2}-\overline f_1\frac{\partial\overline g_2}{\partial z_1}+\overline f_2\frac{\partial g_1}{\partial z_1}-f_2\frac{\partial\overline g_2}{\partial\overline z_2}=0
\\
&&g_1(\frac{\partial
f_1}{\partial\overline z_2}-\frac{\partial\overline f_2}{\partial z_1})+f_1\frac{\partial g_1}{\partial\overline z_2}+\overline f_1\frac{\partial\overline g_2}{\partial z_1}-\overline f_2\frac{\partial g_1}{\partial z_1}-f_2\frac{\partial\overline g_2}{\partial\overline z_2}=0
\end{eqnarray*}
\end{proof}

%%%%%%%%%%%%%%%%%%%%%%%%%%%%%%%%%%%%%%%%%%%%%%%%%%%%%%%%%%%%%%%%%%%%%%
\subsection{Definition} We will call {\it hypermeromorphic} the w-hypermeromorphic functions \\whose sum and product are w-hypermeromorphic. Their space is nonempty, since it contains the space of the meromorphic functions.

%%%%%%%%%%%%%%%%%%%%%%%%%%%%%%%%%%%%%%%%%%%%%%%%%%%%%%%%%%%%%%%%%%%%%%
%%%%%%%%%%%%%%%%%%%%%%%%%%%%%%%%%%%%%%%%%%%%%%%%%%%%%%%%%%%%%%%%%%%%%%
\section{Globalisation. Hamilton 4-manifold}\label{Sec5}

%%%%%%%%%%%%%%%%%%%%%%%%%%%%%%%%%%%%%%%%%%%%%%%%%%%%%%%%%%%%%%%%%%%%%%
\subsection{\hspace*{-8pt}}The hypermeromorphic functions on a relatively compact open set $U$ of $\H$ play the part of the meromorphic functions on a relatively compact open set $U$ of $\C$. 
We will call {\it pseudoholomorphic function on $U$}, every hypermeromorphic function, without poles on $U$. We will call {\it smooth hypermeromorphic function (sha function) on $U$}, every hypermeromorphic function, without zeroes and poles on $U$. 
%---------------------------------------------------------------------
\begin{lem}{\it The quotient of two pseudoholomorphic functions on $U$, with the same zeroes and the same orders, is a sha function on $U$.
}\end{lem}
%---------------------------------------------------------------------

%%%%%%%%%%%%%%%%%%%%%%%%%%%%%%%%%%%%%%%%%%%%%%%%%%%%%%%%%%%%%%%%%%%%%%
\subsection{Manifolds}
The sha functions have been defined on open sets of $\H\cong \C^2$. Let $X$ be a 4-dimensional manifold bearing an atlas ${\mathcal A}$ of charts $(h_j,U_j)$ such as the transition functions $h_{i,j}: U_i\cap U_j\rightarrow \H$ are sha functions. $X=(X,\mathcal A)$ will be called an ${\mathcal A}$-manifold analogous for $\H$ of a Riemann surface for $\C$. I also propose to call an ${\mathcal A}$-manifold a {\it Hamilton 4-manifold}.

%%%%%%%%%%%%%%%%%%%%%%%%%%%%%%%%%%%%%%%%%%%%%%%%%%%%%%%%%%%%%%%%%%%%%%
\subsection{Sheaves of pseudoholomorphic, hypermeromorphic functions}

%%%%%%%%%%%%%%%%%%%%%%%%%%%%%%%%%%%%%%%%%%%%%%%%%%%%%%%%%%%%%%%%%%%%%%
\subsubsection{Functions on an $\mathcal A$-manifold $X=(X,\mathcal A)$} 
A map $f:X\rightarrow \H$ is called a {\it pseudoholomorphic function} on $X$, if it is continuous and satisfies the following condition:
for every chart $(h,U)$ of $X$, 
$(f\vert U)h^{-1}: h(U)\rightarrow \H)$ 
is a pseudoholomorphic. 
In the same way, a map $f:X\rightarrow \H$ is called a {\it hypermeromorphic function} on $X$, if it is continuous and satisfies the following condition:for every chart $(h,U)$ of $X$, 
$(f\vert U)h^{-1}: h(U)\rightarrow \H)$ 
is a hypermeromorphic.
%%%%%%%%%%%%%%%%%%%%%%%%%%%%%%%%%%%%%%%%%%%%%%%%%%%%%%%%%%%%%%%%%%%%%%
\subsubsection{Examples of Hamilton 4-manifold} The identity map of $\H$ is: $z_1+z_2{\bf j}\mapsto z_1+z_2{\bf j}$.

\noindent\underline{Ex. 1:} $(id_{\H},\H)$ is the unique chart of the atlas defining $\H$ as an $\mathcal A$-manifold.
{\it Proof}. The identity map $(id_{\H}$ is $f_1=z_1, f_2=z_2$ is pseudoholomorphic. 

\noindent\underline{Ex.2:} Every open set $V$ of $X$ bears an induced structure of {\it Hamilton 4-manifold}. 

\noindent\underline{Ex. 3:} {\it Hamilton hypersphere $\P$.} 

In the space $\H$x$\H\setminus\{0\}$, consider the equivalence relation $\rho_1 {\mathcal R}\rho_2$: there exists $\lambda\in\H^*=\H\setminus 0$ such that $\rho_2=\rho_1\lambda$ (right multiplication by $\lambda$). The elements of $\H$x$\H\setminus\{0\}$ are the pairs $(q_1,q_2)\not =(0,0)$. Let 

$\pi: \H$x$\H\setminus\{0\}\rightarrow \big(\H$x$\H\setminus\{0\}\big)/{\mathcal R}$ denoted $\P.$ 

\hskip 12mm $(q_1,q_2)\mapsto$ class of $(q_1,q_2)$

So, $\P$ is the set of the quaternionic lines from the origin of $\H^2$.

Consider the pairs $(q_1,q_2)\in \H^2$, with $q_2\not=0$ we have: $\pi(q_1,q_2)=\pi(q_1q_2^{-1},1)$; let $\zeta=q_1q_2^{-1}, q_2\not=0$; in the same way, consider the pairs $(q_1,q_2)\in \H^2$, with $q_1\not=0$ we have: $\pi(q_1,q_2)=\pi(1,q_2q_1^{-1})$; let $\zeta'=q_2q_1^{-1}, q_1\not=0$. The charts $\zeta, \zeta'$ have for domains $U$, $U'$, two open sets of $\P$, respectively homeomorphic to $\H$ forming an atlas of $\P$. Remark that $U$ covers the whole of $\P$ except the point $\pi(q_1,0)$ denoted $\infty$, and that $U'$ covers the whole of $\P$ except the point $\pi(0,q_2)$ denoted $0$. $U'=\P\setminus \{0\}$. Over $U\cap U'$, we have: $\zeta.\zeta'=1$, i.e. $\zeta'=\zeta^{-1}$ and $\zeta=q_1q_2^{-1}$. 

%%%%%%%%%%%%%%%%%%%%%%%%%%%%%%%%%%%%%%%%%%%%%%%%%%%%%%%%%%%%%%%%%%%%%%
\subsubsection{\hspace*{-8pt}} {\it Pseudoholomorphic map} or {\it morphism.} 

Let $X$ and $Y$ be two Hamilton 4-manifolds, a map $f: X\rightarrow Y$ is said pseudoholomorphic if it is continuous and if, for every pair of pseudoholomorphic charts $(h,U), (k,V)$ such that $f(U)\subset V$, 

$k¡(f\vert U)¡h^{-1}:h(U)\rightarrow k(V)$ be pseudoholomorphic.

%%%%%%%%%%%%%%%%%%%%%%%%%%%%%%%%%%%%%%%%%%%%%%%%%%%%%%%%%%%%%%%%%%%%%%
\subsubsection{Sheaf of pseudoholomorphic functions} Let $U, V$ be two open sets of $X$ such that $U\subset V$, then, the restrictions to $U$ of the pseudoholomorphic functions on $V$ are pseudoholomorphic on $U$. 

So is defined the {\it sheaf}, denoted $\mathcal P$, of (non commutative rings) {\it of pseudoholomorphic functions on $X$}. The pair (X,$\mathcal P$) is a {\it ringed space}.

In the same way, the {\it sheaf} of non commutative rings, denoted $\mathcal M$, {\it of hypermeromorphic functions is defined on $X$}.

%%%%%%%%%%%%%%%%%%%%%%%%%%%%%%%%%%%%%%%%%%%%%%%%%%%%%%%%%%%%%%%%%%%%%%
\subsubsection{\it Hamiltonian Submanifolds}They are submanifolds whose function ring is pseudoholomorphic. We will implicitly use the following fact: If $f$ is a pseudoholomorphic or hypermeromorphic function, the same is true for $a+f$, where $a$ is any fixed quaternion. 

The following examples are complex analytic submanifolds. 

\noindent i) $\H$. 
Let $a$ be a fixed quaternion, then $a+\C\subset\H$ {\it is a complex line from $a$ embedded in}~$\H$. 

\noindent ii) $\P$. 
{\it Complex projective line imbedded in} $\P$.
Let $i: z_1\mapsto z_1+z_2{\bf j}$ and $j:\CP\mapsto\P$

$\C\times\C\setminus 0\hskip 10mm\rightarrow \H\times \H\setminus\{0\}$ 

$ i\times i \downarrow \hskip 30mm \downarrow$

$\big(\C\times\C\setminus 0\big)/{\mathcal R}' \rightarrow \big(\H\times\H\setminus\{0\}\big){\mathcal R}$

Let $p\in\P$ be a fixed point. Then, $p+{\C P}^1$ is a {\it complex projective line} (or Riemann sphere) {\it from $p$, embedded in} $\P$.

\noindent iii) Let $S$ be a compact Riemann surface contained in $\P$ as a Hamiltonian submanifold. Then $p+S$ is a {\it compact Riemann surface from $p$, embedded in} $\P$.

%%%%%%%%%%%%%%%%%%%%%%%%%%%%%%%%%%%%%%%%%%%%%%%%%%%%%%%%%%%%%%%%%%%%%%
\subsubsection{\it A family of complex submanifolds in a Hamilton 4-manifold} We now use the properties and notions of subsection~\ref{Subsec:2.5.2}. 
They are $a+A_{\alpha, \beta}$ and also for restrtictions to an open set $U$ of $\H$.

On a Hamilton 4-manifold $X$ with an atlas $\mathcal
A$ and every domain of chart $U$ as above, we obtain:
\prop{\it Let ($X,\mathcal P$) be a Hamilton 4-manifold. There exist a family of complex analytic curves $C_{b,\gamma,\delta}$, of $X$. For every $U$ domain of coordinates in $\mathcal A$ let $A_{\gamma,\delta}$. By gluing, we get a complex analytic curve in $(X,{\mathcal P})$ from $b\in X$, and $\gamma$, $\delta$ are real parameters.
}

\begin{proof} 
Let $b\in X; \beta,\gamma\in \R$ be given, consider an atlas $\mathcal A$ whose domains of charts are either open sets $U$ of $X$ disjoint from $A_{\beta,\gamma}$, or $V_{\beta,\gamma}=U\cup ( A_{\beta,\gamma}\cap U)$
where $A_{\beta,\gamma}\cap U$ is connected, not empty.
The restrictions of the charts of $\mathcal A$ to the $U\cup ( A_{\beta,\gamma}\cap U)$ define an atlas of $C_{b,\gamma,\delta}$ as complex analytic subvariety of $(X,{\mathcal P})$, in the following way: assume $b\in V_{\beta,\gamma}\cap A_{\beta,\gamma}\cap U$ and consider the open sets analogous to $V_{\beta,\gamma}$ such that the various $V_{\beta,\gamma}$ be connected. Then the corresponding $A_{\beta,\gamma}\cap U$ constitute a covering of the unique complex analytic curve $C_{b,\gamma,\delta}$. \end{proof} 

%%%%%%%%%%%%%%%%%%%%%%%%%%%%%%%%%%%%%%%%%%%%%%%%%%%%%%%%%%%%%%%%%%%%%%
\subsubsection{\hspace*{-8pt}}{Let $C$ be a complex analytic curve embedded into $X$ and an atlas $\mathcal A$ such that every chart of domain $U$ meeting $C$ satisfies: $U\cap C$ is connected}

%---------------------------------------------------------------------
\thm{\it The set of complex analytic curves in $X$ is the family $C_{b,\gamma,\delta}$.} 
%---------------------------------------------------------------------
\rm

%%%%%%%%%%%%%%%%%%%%%%%%%%%%%%%%%%%%%%%%%%%%%%%%%%%%%%%%%%%%%%%%%%%%%%
%%%%%%%%%%%%%%%%%%%%%%%%%%%%%%%%%%%%%%%%%%%%%%%%%%%%%%%%%%%%%%%%%%%%%%
\section{Hamilton 4-manifold of a hypermeromorphic function}\label{Sec6}

%%%%%%%%%%%%%%%%%%%%%%%%%%%%%%%%%%%%%%%%%%%%%%%%%%%%%%%%%%%%%%%%%%%%%%
\subsection{Analytic continuation along a path} \cite[p.~116]{D90}

Let $X$ be a Hamilton 4-manifold, $\gamma :[0,1]\rightarrow X$ a continuous path from $a$ to $b$, $\varphi\in{\mathcal P}_a$ a germ of pseudoholomorphic function at $a$.

Let $\tau\in [0,1]$ and $\varphi_\tau\in{\mathcal P}_{\gamma(\tau)}$, there exists an open neighborhood $U_\tau$ of $\gamma(\tau)$ in $X$ and a pseudoholomorphic function $f_\tau\in{\mathcal P
}(
U_\tau)$ such that $\rho^{U_\tau}_{\gamma(\tau)}f_\tau=\varphi_\tau$. $\gamma$ being continuous, it exists an open neighborhood $W_\tau$ of $\tau$ in [0,1] such that $\gamma (W_\tau)\subset U_\tau$. 

%%%%%%%%%%%%%%%%%%%%%%%%%%%%%%%%%%%%%%%%%%%%%%%%%%%%%%%%%%%%%%%%%%%%%%
\subsection{Definition} 

A germ $\psi\in{\mathcal P}_b$ is said to be {\it the analytic continuation of $\varphi$ along $\gamma$} if there exists a family $(\varphi_t)_{t\in [0,1]}$ such that:

1) $\varphi_0=\varphi$ and $\varphi_1=\psi$.

2) for every $\tau\in [0,1]$, for every $t\in W_\tau$, we have: $\rho^{U_\tau}_{\gamma(\tau)}f_\tau=\varphi_\tau$

%---------------------------------------------------------------------
\begin{thm}{\it Identity theorem. Let $X$ be a connected Hamilton 4-manifold and $f_1, f_2: X\rightarrow Y$ be two morphisms which coincide in the neighborhood of a point $x_0\in X$, then $f_1, f_2$ coincide on $X$.}\end{thm}
%---------------------------------------------------------------------

Proof as for Riemann surfaces, \cite[ch.~5]{D90}.
%---------------------------------------------------------------------
\begin{thm}{\it Let X be a simply connected Hamilton 4-manifold, $a\in X$, $\varphi\in{\mathcal P}_a$ be a germ having an analytic continuation along every path from $a$. Then there exists a unique function $f\in {\mathcal P}(X)$ such that $\rho^X_af=\varphi$.}\end{thm}
%---------------------------------------------------------------------
(cf.~\cite[ch.~5,~4.1.5]{D90})

Let $p:Y\rightarrow X$ be a morphism of two Hamilton 4-manifolds; $p$ is locally bi-pseudoholo\-morphic, then it defines, for every $y\in Y$, an isomorphism $p_y^¡:{\mathcal P}_{x,p(y)}\rightarrow {\mathcal P}_{Y,y}$; this defines: $p_*=p_{*y}=(p^*_y)^{-1}$.

%%%%%%%%%%%%%%%%%%%%%%%%%%%%%%%%%%%%%%%%%%%%%%%%%%%%%%%%%%%%%%%%%%%%%%
\subsection{Definition} 
Let $X$ be a Hamilton 4-manifold, $a\in X$, $\varphi\in{\mathcal P}_a$. A quadruple $(Y,p,f,b)$ is called an {\it analytic continuation of} $\varphi$ if: 

(i) $Y$ is a Hamilton 4-manifold, $p: Y\rightarrow X$ is a morphism;

(ii) $f$ is a pseudoholomorphic function on $Y$; 

(iii) $b\in p^{-1}(a)\subset Y$; $p_*(\rho^Y_bf)=\varphi$.

An analytic continuation is said to be {\it maximal} if it is solution of the following universal map problem: for every analytic continuation $(Z,q,g,c)$ of $\varphi$, there exists a fibered morphism $F: Z\rightarrow Y$ such that $F(c)=b$ and $F^*(f)=g$. Hence 

{\it If $(Y,p,f,b)$ is a maximal analytic continuation of $\varphi$, it is unique up to an isomorphism. $Y$ is called the Hamilton 4-manifold of $\varphi$}. 
%---------------------------------------------------------------------
\begin{thm}{\it Let $X$ be a Hamilton 4-manifold, $a\in X$, $\varphi\in{\mathcal P}_a$. Then there exists a maximal analytic continuation of $\varphi$.}\end{thm} 
%---------------------------------------------------------------------

%%%%%%%%%%%%%%%%%%%%%%%%%%%%%%%%%%%%%%%%%%%%%%%%%%%%%%%%%%%%%%%%%%%%%%
\subsection{Remark}
Then, we will say that the above function $f$ is {\it the unique maximal analytic continuation of the germ $\varphi$}. Moreover, the above definitions and results of the section~\ref{Sec2} are valid for the sheaf ${\mathcal M}$ of hypermeromorphic functions instead of the sheaf~${\mathcal P}$.

%%%%%%%%%%%%%%%%%%%%%%%%%%%%%%%%%%%%%%%%%%%%%%%%%%%%%%%%%%%%%%%%%%%%%%
\subsection{Main result}
%---------------------------------------------------------------------
\begin{thm}{\it \label{Thm2.4}Let $X$ be a Hamilton 4-manifold and $P(T)= T^n+c_1T^{n-1}+\ldots+c_n\in {\mathcal M}(X)[T]$ be an irreducible polynomial of degree $n$. Then there exist a Hamilton 4-manifold $Y$, a ramified pseudoholomorphic covering (cf.~\cite[ch.~5]{D90} for Riemann surfaces) with $n$ leaves $\Pi: Y\rightarrow X$ and a hypermeromorphic function $F\in {\mathcal M}(Y)$ such that $(\Pi^*P)(F)=0$.}\end{thm} 
%---------------------------------------------------------------------
$F$ is the unique maximal analytic continuation of every hypermeromorphic germ $\varphi$ of $X$ such that $P(\varphi)=0$; $F$ is called the {\it hyperalgebraic function defined by the polynomial $P$} and $Y$ is the {\it Hamilton 4-manifold of} $F$. 

{\it Proof} at the end of the section.
\\
1) $X$ {\it is compact connected}.
\\
2) {\it Every pseudoholomorphic function on $X$ is constant.}
\\
3) {\it Every hypermeromorphic function $f$ on $X$ different from $\infty$ is rational}. 
\\
4) {\it In case $X=\P$, in Theorem~\ref{Thm2.4}, $c_j$ is rational.} Indeed, since $c_j$ is hypermeromorphic, from 3), it is rational.

%%%%%%%%%%%%%%%%%%%%%%%%%%%%%%%%%%%%%%%%%%%%%%%%%%%%%%%%%%%%%%%%%%%%%%
\subsection{Proof of Theorem~\ref{Thm2.4}} In the notations of Ex. 3, $\zeta$ is a local coordinate on $X=\P$. 

$f$ has a finite set of poles $p_1,\ldots, p_n$. Assume that $\infty$ is not a pole of $f$, then $p_1,\ldots, p_n\in\H$. Let $h_\nu$ the principal part of $f$ at $p_\nu$, then $f-h_\nu=a_\nu$, constant, from 2) and $h_\nu=\displaystyle\sum_{j=-k_\nu}^{-1}C_{\nu j}(\zeta-p_\nu^j)$ is a hypermeromorphic function, where $C_{\nu j}\in\H$.

%%%%%%%%%%%%%%%%%%%%%%%%%%%%%%%%%%%%%%%%%%%%%%%%%%%%%%%%%%%%%%%%%%%%%%
\subsubsection{Elementary symmetric functions} 

Let $$\Pi:Y\rightarrow X$$ be a nonramified pseudoholomorphic covering with $n$ leaves, and $f$ be a hypermeromorphic function on $Y$. Every point $x\in X$ has an open neighborhood $U$ such that $\Pi^{-1}(U)=\displaystyle\bigcup_{j=1}^n V_j$ where the $V_j$ are disjoint and $\Pi\vert V_j: V_j\rightarrow U$ is bi-pseudoholomorphic, $(j=1,\ldots,n)$; let $\varphi_j:U\rightarrow V_j$ the reverse (i.e. set inverse) of $\Pi\vert V_j$ and $f_j=\varphi_j^*f=f.\varphi_j$. Then:

$$\Pi_{j=1}^n (T-f_j)=T^n+c_1T^{n-1}+\ldots+ c_n;$$ 
$c_j=(-1)^js_j(f_1,\ldots,f_ n)$, where $s_j$ is the $j$-th elementary symmetric function in $n$ variables. The $c_j$ are hypermeromorphic, locally defined, but glue together into $c_1,\ldots,c_n\in {\mathcal M} (X)$ and are called {\it the elementary symmetric functions of $f$ with respect to $\Pi$}.

%%%%%%%%%%%%%%%%%%%%%%%%%%%%%%%%%%%%%%%%%%%%%%%%%%%%%%%%%%%%%%%%%%%%%%
\subsubsection{Remark}
The elementary symmetric functions of a hypermeromorphic function on $Y$ are still defined when the covering $\Pi$ is ramified.

%%%%%%%%%%%%%%%%%%%%%%%%%%%%%%%%%%%%%%%%%%%%%%%%%%%%%%%%%%%%%%%%%%%%%%
\subsubsection{\hspace*{-8pt}}
%---------------------------------------------------------------------
\begin{thm}{\it Let $\Pi$ as in Theorem~\ref{Thm2.4}, with $Y$ not necessarily connected, $A\subset X$ be a discreet closed subset containing all the critical values of $\Pi$, and $B=\Pi^{-1}(A)$.

Let $f$ be a pseudoholomorphic (resp. hypermeromorphic) function on $Y\setminus B$ and
$$c_1,\ldots,c_n\in {\mathcal H} (X\setminus A) (resp.{\mathcal M} (X\setminus A))$$
the elementary symmetric functions of $f$. Then the following two conditions are equivalent:

$(i)$ $f$ has a pseudoholomorphic (resp. hypermeromorphic) extension to $Y$;

$(ii)$ for every $j=1,\ldots,n$, $c_j$ has a pseudoholomorphic (resp. hypermeromorphic) extension to $X$.
}\end{thm}
%---------------------------------------------------------------------

%%%%%%%%%%%%%%%%%%%%%%%%%%%%%%%%%%%%%%%%%%%%%%%%%%%%%%%%%%%%%%%%%%%%%%
\subsubsection{Existence of $Y$ in \texorpdfstring{Theorem~\ref{Thm2.4}}{Theorem 6.4}}

Let $\Delta\in {\mathcal M}(X)$ be the discriminant of $P(T)$; $P(T)$ being irreducible, $\Delta\not =0$: then there exists a discrete closed set $A\subset X$ such that, for every $x\in X'=X\setminus A$, $\Delta (x)\not=0$, and all the functions $c_j$ are pseudoholomorphic. 

Let $Y'=\{\varphi\in{\mathcal H}_x, x\in X'; P(\varphi)=0\}\subset L\mathcal H$, etal space defined by the sheaf $\mathcal H$, and $\Pi': Y'\rightarrow X$, $(\varphi\mapsto x)$. 

It can be shown that, for every $x\in X'$, there exists an open neighborhood $U$ of $x$ in $X'$ and functions $f_j\in {\mathcal H}(U)$, $j=1,\ldots, n$, such that $P(T)\vert U=\displaystyle\Pi_{j=1}^n(T-f_j)$; then $\Pi'^{-1}(U)=\displaystyle\bigcup_{j=1}^n[U,f_j]$ where $[U,f_j]=\{f_{jy},y\in U\}$ is an open set of $L{\mathcal H}$ and $\Pi'\vert [U,f_j]:[U,f_j]\rightarrow U$ is a homeomorphism; $Y'$ is a Hamilton 2-manifold non necessarily connected, and a pseudoholomorphic, non ramified covering of $X'$. It can be shown that $\Pi'$ can be extended into a ramified pseudoholomorphic covering $\Pi :Y\rightarrow X$ of $X$ for which $Y'=\Pi^{-1}(X').$ 

The $c_j$ are defined on the whole of $X$; from Theorem 6.5, $f$ has an extension $F\in {\mathcal M}(X)$ such that 

$$\Pi^*P(F)=F^n+(\Pi^*c_1)F^{n-1}+\ldots+\Pi^*c_n=0.$$ 

It is easy to prove the connectedness of $Y$ and the unicity of $F$.

This ends the proof of Theorem~\ref{Thm2.4}.

%%%%%%%%%%%%%%%%%%%%%%%%%%%%%%%%%%%%%%%%%%%%%%%%%%%%%%%%%%%%%%%%%%%%%%
%%%%%%%%%%%%%%%%%%%%%%%%%%%%%%%%%%%%%%%%%%%%%%%%%%%%%%%%%%%%%%%%%%%%%%
\section{The Hamilton 4-manifold \texorpdfstring{$Y$ of $F$ when $X=\P$}{Y of F when X=\P}}\label{Sec7}

%%%%%%%%%%%%%%%%%%%%%%%%%%%%%%%%%%%%%%%%%%%%%%%%%%%%%%%%%%%%%%%%%%%%%%
\subsection{Recall the main properties of $Y$} \hspace*{1pt}

$Y$ is of real dimension 4; 

$Y$ is connected;

$Y$ is compact; 

$Y$ is $C^\infty$; 

let $m$ be the number of the critical values of $\Pi$ and $q_j$ these critical values; they define points of $Y$ forming the $0$-skeleton of a simplicial complex $K$ carried by the manifold $Y$. $K$ may be supposed to be $C^\infty$ by parts. Cutting along the $3$-faces of $K$ defines a fundamental domain $FD$ of the covering $\Pi$. $FD$ is a 4-dim polytope in $\P$ with an even number of 3-faces; gluying together the opposite 3-faces, we get a compact 4-dim polytope with homology of the Hamilton 4-manifold $Y$. 

%%%%%%%%%%%%%%%%%%%%%%%%%%%%%%%%%%%%%%%%%%%%%%%%%%%%%%%%%%%%%%%%%%%%%%
\subsection{Homology of $Y$} $H^p(Y;\Z)$, for $p=0,\ldots,4$ have to be evaluated, using the critical values $q_j$, and the Poincar\'e duality.

%%%%%%%%%%%%%%%%%%%%%%%%%%%%%%%%%%%%%%%%%%%%%%%%%%%%%%%%%
%%%%%%%%%%%%%%%%%%%%%%%%%%%%%%%%%%%%%%%%%%%%%%%%%%%%%%%%%


\begin{thebibliography}{CLSSS007}

\bibitem[CSSS04]{CSSS 04} 
F. Colombo, I. Sabadini, F. Sommen, D.C. Struppa, \emph{Analysis of Dirac Systems and Computational Algebra}, Progress in Math. Physics {\bf 39}, Birkh\"auser (2004). %82ANA04

\bibitem[CLSSS07]{CLSSS 07} F. Colombo, E. Luna-Elizarrar\'as, I. Sabadini, M.V. Shapiro, D.C. Struppa, \emph{A new characterization of pseudoconvex domains in $\C^2$}, Comptes Rendus Math. Acad. Sci. Paris, 344 (2007), 677-680. 

\bibitem[D90]{D90} P.~Dolbeault, \emph{Analyse complexe}, Collection Ma\^{\i}trise de math\'e\-matiques pures, Masson Paris 1990; 
%\'ed. corrig\'ee num\'eris\'ee 2011, en libre acc\`es.
open access digitalized version available at\newline \textsf{\small\href{http://pierre.dolbeault.free.fr/Book/P_Dolbeault_Analyse_Complexe.pdf}{http://pierre.dolbeault.free.fr/Book/P$\underline{\phantom x}$Dolbeault$\underline{\phantom x}$Analyse$\underline{\phantom x}$Complexe.pdf}}

\bibitem[D13]{D13} P. Dolbeault, \emph{On quaternionic functions}, \textsf{\small\href{http://arxiv.org/abs/1301.1320}{arXiv:1301.1320}}.

\bibitem[F39]{F 39} R. Fueter, \emph{\"Uber einen Hartogs'schen Satz}, Comm. Math. Helv. 12 (1939), 75-80.

\bibitem[M66]{M 66} B. Malgrange, \emph{Ideals of differentiable functions}, Oxford Univ. Press (1966).

\bibitem[R13]{R 13} G. Roos, Personnal communication, March 2013.

\end{thebibliography}
\end{document}